\documentclass[12pt,a4paper]{article}
\usepackage{amssymb,latexsym,amsmath}
\usepackage{float}              % H insere figura exatamente no local

\usepackage{epsfig}

\usepackage{subfig}
\usepackage{graphicx}

\usepackage{xcolor}

\input epsf

\newtheorem{prop}{Proposition}[section]
\newtheorem{theo}[prop]{Theorem}
\newtheorem{cor}[prop]{Corollary}
\newtheorem{ex}[prop]{Example}
\newtheorem{defn}[prop]{Definition}

\newenvironment{proof}
 {\begin{trivlist} \item[\hskip \labelsep {\bf Proof}\hspace*{3 mm}]}
 {\hfill$\Box$\end{trivlist}}

\begin{document}

\title{Curves in a spacelike hypersurface in the Minkowski space-time}
\author{Shyuichi Izumiya, Ana Claudia Nabarro\footnote{Supported by  FAPESP grant
 2016/19139-7} and  Andrea de Jesus Sacramento\footnote{Supported by  CNPq grant
 150469/2017-9}}

 \maketitle
\begin{abstract}
We define the hyperbolic surface  and the de Sitter surface of a curve in the spacelike hypersurface $M$ in the Minkowski $4$-space. These surfaces are respectively located   in the hyperbolic 3-space and in the de Sitter 3-space.  We use  techniques of the theory of singularities in order to describe  the generic shape of these surfaces and of their singular value sets. We also investigate geometric meanings of those singularities.
\end{abstract}

\renewcommand{\thefootnote}{\fnsymbol{footnote}}
\footnote[0]{2010 Mathematics Subject classification 58K05, 53D10, 53B30.}

%\renewcommand{\thefootnote}{\fnsymbol{footnote}}
%\footnote[0]{2010 Mathematics Subject classification 58K05, 53D10, 53B30.}
%\footnote[0]{Key Words and Phrases. Curves on a timelike surface, pseudo-spherical evolute,
%Minkowski space, lightlike surface.}

%%%%%%%%%%%%%%%%%%%%%%%%%%%%%%%%%%%%%%%%%%%%%%%%%%%%%%%%%%%%%%

\section{Introduction}\label{sec:intro}

Submanifolds in Lorentz-Minkowski space are investigated from various mathematical viewpoints  and  are of interest also in relativity theory. In recent years, using singularity theory, very important progress has been made and many investigations have been conducted  to  classify and characterize the singularity of submanifolds in Euclidean spaces or in semi-Euclidean spaces (see, for example, \cite{bruce}-\cite{anaandrea} and \cite{satopseudo}).  The results in this paper contribute to the study of the extrinsic geometry of curves in different ambient spaces.

We consider a spacelike embedding $X:U\rightarrow\mathbb{R}^4_1$ from an open subset $U\subset\mathbb{R}^3$ and identify $M$ and $U$ through the embedding $X$, where $\mathbb{R}^4_1$ is the  Minkowski $4$-space. For a curve  $\gamma:I\rightarrow M$ with nowhere vanishing curvature, we define a hyperbolic surface in hyperbolic space $H^{3}(-1)$ and a de Sitter  surface in the de Sitter space $S^{3}_1$ associated to the curve $\gamma$. For the study of the generic differential geometry  of these surfaces and of their singular  sets, we use singularity theory techniques,  and in particular, the classical deformation theory.

Our paper is organized as follows: In Section \ref{sec:pre}, we  review basic definitions  for the Minkowski 4-space and construct a  moving frame along $\gamma$ together with Frenet-Serret type formulae. We also review the definition of the $A_k$-singularities and discriminant sets.
 In  Sections  \ref{sectionht} and \ref{sectionhs}, we define two  families of height functions on $\gamma$, which are timelike tangential height functions and  spacelike tangential height functions. These functions measures the contact of the curve $t$ with special hyperplanes.  Differentiating these functions yield invariants related to each surface.  We show  that the hyperbolic surface of $\gamma$ is the discriminant set of the family of  timelike tangential height functions (Corollary \ref{coro:discsuphiper}) and the de Sitter surface of $\gamma$ is the discriminant set of the family of spacelike tangential height  functions (Corollary \ref{coro:discsupdesitter}). Furthermore, using the theory of deformations,  we  give a classification  and a  characterization of the diffeomorphims type  of these surfaces (Theorems \ref{teo:classificacao} and \ref{teo:classificacao2}).
We  also investigate the geometric meaning  of these invariants.  We prove  results that give conditions (related to these invariants) for the curve $\gamma$ to be part of a slice surface (Propositions \ref{pro:slice1} and \ref{pro:slice2}). When $\gamma$ is not part of a slice surface, we characterize the contact of $\gamma$ with a slice surface using the singularity types of its hyperbolic surface (Proposition \ref{prop:contato1} ) and the singularity types of  its de Sitter surface (Proposition \ref{prop:contato2}).  In  Sections  \ref{sec:exh3} and \ref{sec:exas3}, we consider examples of curves on spacelike hypersurface in $\mathbb{R}^4_1$ and we obtain the  surfaces studied in \cite{izumyiafs}.
%curves on $\mathbb{R}^3$ and curves on the  hyperbolic space, $H^3(-1)$.

\section{Preliminaries}\label{sec:pre}
 The \emph{Minkowski space} $\mathbb{R}^{4}_1$ is the vector space $\mathbb{R}^{4}$ endowsed with the pseudo-scalar product $\langle x,y\rangle=-x_0y_0 + x_1y_1 + x_{2}y_{2}+x_{3}y_{3},$ for any $x=(x_0,x_1,x_{2},x_3)$ and $y=(y_0,y_1,y_{2},y_3)$ in $\mathbb{R}^{4}_1$ (see, e.g., \cite{ratcliffe}). We say that a non-zero vector $x\in\mathbb{R}^{4}_1$ is \emph{spacelike} if $\langle x,x\rangle > 0$, \emph{lightlike} if $\langle x,x\rangle = 0$ and
 \emph{ timelike} if $\langle x,x\rangle < 0$, respectively. We say that $\gamma:I \rightarrow \mathbb{R}^4_1$, with $I\subset \mathbb{R}$ an open interval, is \emph{spacelike} (resp. \emph{timelike}) if the tangent vector $\gamma'(t)$ is a \emph{spacelike }(resp. \emph{timelike}) vector for any $t\in I$.
  %A point $\gamma(t)$ is called a \emph{lightlike point} if $\gamma'(t)$ is a lightlike vector.
The norm of a vector $x\in\mathbb{R}^{4}_1$ is defined by $\parallel x\parallel=\sqrt{\mid\langle x,x\rangle\mid}$. For a non-zero vector $v\in\mathbb{R}^{4}_1$ and a real number $c$, we define a \emph{hyperplane} with \emph{pseudo-normal} $v$ by $$HP(v,c)=\{x\in\mathbb{R}^{4}_1\,\,|\,\,\langle x,v\rangle=c\}.$$ We call $HP(v,c)$ a \emph{spacelike hyperplane}, a \emph{timelike hyperplane} or a \emph{lightlike hyperplane} if $v$ is timelike, spacelike or lightlike, respectively. We now consider the pseudo-spheres in $\mathbb{R}^{4}_1$: The \emph{hyperbolic 3-space} is defined  by
 $$H^{3}(-1)=\{x\in\mathbb{R}^{4}_1\,\,|\,\,\langle x,x\rangle=-1\}$$ and the de \emph{Sitter 3-space} by
 $$S^{3}_1=\{x\in\mathbb{R}^{4}_1\,\,|\,\,\langle x,x\rangle=1\}.$$

 For any $x=(x_0,x_1,x_2,x_3)$, $y=(y_0,y_1,y_2,y_3)$, $z=(z_0,z_1,z_2,z_3)\in\mathbb{R}^{4}_1$, the pseudo vector product of $x$, $y$ and $z$ is defined as follows:
 $$x\wedge y\wedge z=\left|\begin{array}{cccc}
                                                                                        -e_0 & e_1 & e_2 & e_3 \\
                                                                                        x_0 & x_1 & x_2 & x_3 \\
                                                                                        y_0 & y_1 & y_2 & y_3\\
                                                                                        z_0 & z_1 & z_2 & z_3
                                                                                      \end{array}\right|,$$
where $\{ e_0, e_1, e_2,e_3\}$ is the canonical basis of $\mathbb{R}^{4}$.

 We consider a spacelike embedding $X:U\rightarrow\mathbb{R}^4_1$ from an open subset $U\subset\mathbb{R}^3$. We write  $M=X(U)$ and identify $M$ and $U$ through the embedding $X$. We say that $X$ is a \emph{spacelike embedding} if the tangent space $T{_p}M$ consists of spacelike vectors at any $p=X(u)$. Let $\bar{\gamma}:I\rightarrow U$ be a regular curve. Then we have a curve $\gamma:I\rightarrow M\subset\mathbb{R}^4_1$ defined by $\gamma(s)=X(\bar{\gamma}(s))$. We say that $\gamma$ is a \emph{curve in the spacelike hypersurface $M$}. Since $\gamma$ is a spacelike curve, we can reparametrize it by the arc length s,  then
 we have the unit tangent vector  $t(s)=\gamma'(s)$. In this case, we call $\gamma$ a {\it unit speed spacelike curve}. Since $X$ is a spacelike embedding, we have a unit timelike normal vector field $n$ along $M=X(U)$ defined by $$n(p)=\frac{X_{u_1}(u)\wedge X_{u_2}(u)\wedge X_{u_3}(u)}{\parallel X_{u_1}(u)\wedge X_{u_2}(u)\wedge X_{u_3}(u)\parallel}$$ for $p=X(u)$, where $X_{u_i}={\partial X}/{\partial u_i}$, $i=1,2,3$. We say that $n$ is \emph{future directed} if $\langle n,e_0\rangle<0$. We choose the orientation of $M$ such that $n$ is future directed. We define $n_{\gamma}(s)=n\circ\gamma(s)$, so that we have a unit timelike normal vector field $n_{\gamma}$ along $\gamma$. Under the assumption that $\parallel \langle n_{\gamma}(s),t'(s)\rangle n_{\gamma}(s)+t'(s)\parallel\neq0$, we define $$n_1(s)=\dfrac{\langle n_{\gamma}(s),t'(s)\rangle n_{\gamma}(s)+t'(s)}{\parallel \langle n_{\gamma}(s),t'(s)\rangle n_{\gamma}(s)+t'(s)\parallel}.$$
It follows that $\langle t,n_1\rangle=0$ and $\langle n_{\gamma},n_1\rangle=0$. Therefore, we have a spacelike unit vector defined by $n_2(s)=n_{\gamma}\wedge t(s)\wedge n_1(s)$. Then, we have a pseudo-orthonormal frame $\{n_{\gamma}, t(s), n_1(s), n_2(s)\}$, which is called a \emph{Lorentzian Darboux frame} along $\gamma$. By standard arguments, the Frenet-Serret type formulae for the above frame are given by
$$
\left\{
  \begin{aligned}
     n_{\gamma}'(s) &=k_n(s)\,t(s)+\tau_1(s)\,n_1(s)+\tau_2(s)\,n_2(s), \\
     t'(s) & =k_n(s)\,n_{\gamma}(s)+k_g(s)\,n_1(s),\\
 n_1'(s) &= \tau_1(s)\,n_{\gamma}(s)-k_g(s)\,t(s)+\tau_g(s)\,n_2(s),\\
  n_2'(s) &= \tau_2(s)\,n_{\gamma}(s)-\tau_g(s)\,n_1(s),
  \end{aligned}
\right.$$
\noindent where $k_n(s)=-\langle n_{\gamma}(s),t'(s)\rangle$, $\tau_1(s)=\langle n_1(s),n_{\gamma}'(s)\rangle$, $\tau_2(s)=\langle n_2(s),n_{\gamma}'(s)\rangle$, $k_g(s)=\parallel \langle n_{\gamma}(s),t'(s)\rangle n_{\gamma}(s)+t'(s)\parallel=\parallel -k_n(s)n_{\gamma}(s)+t'(s)\parallel$ and $\tau_g(s)=\langle -n_2'(s),n_1(s)\rangle$. The invariant $k_n$ is called  a normal curvature, $\tau_1$ a first normal  torsion, $\tau_2$ a second normal torsion, $k_g$ a geodesic curvature and $\tau_g$ a geodesic torsion.

By the assumption,  $k_g(s)=\parallel \langle n_{\gamma}(s),t'(s)\rangle n_{\gamma}(s)+t'(s)\parallel\neq0$, so that $k_g(s)>0$.

 \begin{defn}\label{def:contact}
 Let $F:\mathbb{R}^{4}_1\rightarrow \mathbb{R}$  be a submersion and $\gamma:I\rightarrow M$ be a regular curve. We say that $\gamma$ and $F^{-1}(0)$ have contact of order $k$ at $s_0$, if the function  $g(s)=F\circ\gamma(s)$ satisfies $g(s_0)=g'(s_0)=\cdots=g^{(k)}(s_0)=0$ and $g^{(k+1)}(s_0)\neq0$, i.e., $g$ has an $A_k$-singularity at $s_0$.
\end{defn}

Let $G:\mathbb{R}\times\mathbb{R}^r,(s_0,x_0)\rightarrow\mathbb{R}$ be a family of germs of functions. We call $G$ an $r$-parameter deformation of $f$ if $f(s)=G_{x_0}(s)$. Suppose that $f$ has an $A_k$-singularity $(k\geq1)$ at $s_0$. We write  $$j^{(k-1)}(\frac{\partial G}{\partial x_i}(s,x_0))(s_0)=\sum_{j=0}^{k-1}\alpha_{ji}(s-s_0)^j,$$ for $i=1,\ldots, r$. Then $G$ is  a \emph{versal deformation} if the $k\times r$ matrix of coefficients $(\alpha_{ji})$ has rank $k$ $(k\leq r)$ (see \cite{bruce}).

 The \emph{discriminant set} of $G$ is the set
 $$\mathcal{D}_{G}=\left\{x\in\left(\mathbb{R}^r,x_0\right) \,\,\mid \,\,G=\frac{\partial G}{\partial s}=0\,\, at\,\, (s,x)\,\,for\,\,some\,\, s\in(\mathbb{R},s_0)\right\}$$ and the \emph{ bifurcation set} of $G$ is $$\mathcal{B}_{G}=\left\{x\in\left(\mathbb{R}^r,x_0\right) \,\,\mid \,\,\frac{\partial G}{\partial s}=\frac{\partial^2 G}{\partial s^2}=0\,\, at\,\, (s,x)\,\,for\,\,some\,\, s\in(\mathbb{R},s_0)\right\}.$$

\begin{theo}\cite{bruce}\label{teobruce1}
Let $G:\mathbb{R}\times\mathbb{R}^r,(s_0,x_0)\rightarrow\mathbb{R}$ be an $r$-parameter deformation of $f$ such that $f$ has an  $A_{k}$-singularity at $s_0$. Suppose that G is a versal deformation. Then $\mathcal{D}_G$ is locally  diffeomorphic to
\begin{itemize}
\item[(1)] $C\times\mathbb{R}^{r-2}$ if $k=2$,
\item[(2)]  $SW\times \mathbb{R}^{r-3}$ if $k=3$,
\end{itemize}
where $C=\{(x_1,x_2)\,\mid\,x_1^2=x_2^3\}$ is the ordinary cusp and $SW=\{(x_1,x_2,x_3)\,\mid\,x_1=3u^4+u^2v,x_2=4u^3+2uv,x_3=v\}$ is the swallowtail surface.
\end{theo}

In  Sections \ref{sectionht} and \ref{sectionhs}, we use special  families of  functions on curves in $M$ to study  the hyperbolic surface and   the de Sitter surface.  In fact, these surfaces are the discriminant sets of these families.   %We show  that the horospherical surface of $\gamma$ is the discriminant set of the family of  horospherical height functions (Corollary \ref{coro:discsuphoro}) and that the hyperbolic dual surface of $\gamma$ is the discriminant set of the family of hyperbolic height functions (Corollary \ref{coro:discsupdual}).

\section{Timelike tangential height functions}\label{sectionht}

In this section, we introduce the  family of timelike tangential height functions  on a curve in a spacelike hypersurface $M$. Furthermore, we define and study the hyperbolic surface which is  given by the discriminant set of this family.

We define a family of functions on a curve $\gamma:I\rightarrow M\subset\mathbb{R}^4_1$ as follows:
$$H^{T}_t:I\times H^3(-1)\rightarrow \mathbb{R};\,\,\,\,(s,v)\mapsto\langle t(s),v\rangle.$$
We call  $H^{T}_t$ a \emph{family of timelike tangential height functions of $\gamma$}. We denote $({h^{T}_t})_{v}(s)=H^{T}_t(s,v)$ for any fixed $v\in H^3(-1)$. The family $H^{T}_t$ measures the contact of the curve $t$ with spacelike hyperplanes in $\mathbb{R}^4_1$. Generically, this contact can be  of order $k$, $k=1,2,3$.

The conditions that  characterize the $A_k$-singularity, $k=1,2,3$  can be obtained in Proposition \ref{singht}.

Observe that by the proof of (2) in the following proposition, we have  $k_g^2(s)> k_n^2(s)$. So we assume that there exist an interval $I$ such that  $k_g^2(s)> k_n^2(s)$ for $s\in I$. Furthermore, in order to avoid complicated situations, we assume that $(k_n\tau_2+k_g\tau_g)(s)\neq0$ for any $s\in I$.

\begin{prop}\label{singht}
Let $\gamma:I\rightarrow M$ be a  unit speed curve with $k_g(s)\neq0$ and $(k_n\tau_2+k_g\tau_g)(s)\neq0$. Then, we have the following:
\begin{itemize}
\item[(1)] ${({h^{T}_t})}_{v}(s)=0$ if and only if there exist $\mu$, $\lambda$, $\eta \in\mathbb{R}$ such that $-\mu^2+\lambda^2+\eta^2=-1$ and  $v=\mu n_{\gamma}(s)+\lambda n_{1}(s)+\eta n_2(s)$.

\item[(2)] ${({h^{T}_t})}_{v}(s)={({h^{T}_t})}_{v}'(s)=0$ if and only if  there exists $\theta \in\mathbb{R}$ such that $$v=\frac{\cosh \theta}{\sqrt{k_g^2(s)-k_n^2(s)}}\left(k_g(s)n_{\gamma}(s)+k_n(s)n_1(s)\right)+\sinh \theta n_2(s).$$

\item[(3)] ${({h^{T}_t})}_{v}(s)={({h^{T}_t})}_{v}'(s)={({h^{T}_t})}_{v}''(s)=0$ if and only if $$v=\frac{\cosh \theta}{\sqrt{k_g^2(s)-k_n^2(s)}}\left(k_g(s)n_{\gamma}(s)+k_n(s)n_1(s)\right)+\sinh \theta n_2(s)$$ and $\tanh \theta=\dfrac{k_gk'_n+k_g^2\tau_1-k_n^2\tau_1-k_nk'_g}{(k_n\tau_2+k_g\tau_g)\sqrt{k_g^2-k_n^2}}(s)$.

\item[(4)]  ${({h^{T}_t})}_{v}(s)={({h^{T}_t})}_{v}'(s)={({h^{T}_t})}_{v}''(s)={({h^{T}_t})}_{v}^{'''}(s)=0$ if and only if $$v=\frac{\cosh \theta}{\sqrt{k_g^2(s)-k_n^2(s)}}\left(k_g(s)n_{\gamma}(s)+k_n(s)n_1(s)\right)+\sinh \theta n_2(s),$$ $\tanh \theta=\dfrac{k_gk'_n+k_g^2\tau_1-k_n^2\tau_1-k_nk'_g}{(k_n\tau_2+k_g\tau_g)\sqrt{k_g^2-k_n^2}}(s)$ and $\rho(s)=0$, where

    $\rho(s)=((-k_gk''_n-k_gk_n\tau_2^2-2k_gk'_g\tau_1-k_g^2\tau'_1-k_g^2\tau_g\tau_2+2k_nk'_n\tau_1+k_n^2\tau'_1-k_n^2k_g\tau_2+k''_gk_n-k_gk_n\tau_g^2)
    (k_n\tau_2+k_g\tau_g)+(k_gk'_n+k_g^2\tau_1-k_n^2\tau_1-k_nk'_g)(2k'_n\tau_2+k_n\tau_1\tau_g+k_n\tau'_2+2k'_g\tau_g+k_g\tau_1\tau_2+k_g\tau'_g))(s)$.

    \item[(5)]  ${({h^{T}_t})}_{v}(s)={({h^{T}_t})}_{v}'(s)={({h^{T}_t})}_{v}''(s)={({h^{T}_t})}_{v}^{'''}(s)={({h^{T}_t})}_{v}^{(4)}(s)=0$ if and only if $$v=\frac{\cosh \theta}{\sqrt{k_g^2(s)-k_n^2(s)}}\left(k_g(s)n_{\gamma}(s)+k_n(s)n_1(s)\right)+\sinh \theta n_2(s),$$ $\tanh \theta=\dfrac{k_gk'_n+k_g^2\tau_1-k_n^2\tau_1-k_nk'_g}{(k_n\tau_2+k_g\tau_g)\sqrt{k_g^2-k_n^2}}(s)$ and  $\rho(s)=\rho'(s)=0$.
\end{itemize}
\end{prop}
\begin{proof}
By definition ${({h^{T}_t})}_{v}(s)=0$ if and only if $\langle t(s),v\rangle=0$. This is equivalent to
$v=\mu n_{\gamma}(s)+\lambda n_{1}(s)+\eta n_2(s)$, where $\mu$, $\lambda$, $\eta \in\mathbb{R}$ and  $-\mu^2+\lambda^2+\eta^2=-1$
so that (1) follows. For  (2), ${({h^{T}_t})}_{v}(s)={({h^{T}_t})}_{v}'(s)=0$ if and only if
 $v=\mu n_{\gamma}(s)+\lambda n_{1}(s)+\eta n_2(s)$ with $-\mu^2+\lambda^2+\eta^2=-1$  and $\langle t'(s),v\rangle= -\mu k_n+\lambda k_g=0$. This is equivalent to
 $$v=\frac{\cosh \theta}{\sqrt{k_g^2(s)-k_n^2(s)}}\left(k_g(s)n_{\gamma}(s)+k_n(s)n_1(s)\right)+\sinh \theta n_2(s).$$

For (3), ${({h^{T}_t})}_{v}(s)={({h^{T}_t})}_{v}'(s)={({h^{T}_t})}_{v}''(s)=0$ if and only if
 $$v=\frac{\cosh \theta}{\sqrt{k_g^2(s)-k_n^2(s)}}\left(k_g(s)n_{\gamma}(s)+k_n(s)n_1(s)\right)+\sinh \theta n_2(s) \,\,\,\mbox{and}\,\,\, \langle t''(s),v\rangle=0.$$ Since $t''(s)=(k_n^2(s)-k_g^2(s))t(s)+(k'_n(s)+k_g(s)\tau_1(s))n_{\gamma}(s)+(k_n(s)\tau_1(s)+k'_g(s))n_1(s)+(k_n(s)\tau_2(s)+k_g(s)\tau_g(s))n_2(s)$,
  the previous assertion is equivalent to  $$v=\frac{\cosh \theta}{\sqrt{k_g^2(s)-k_n^2(s)}}\left(k_g(s)n_{\gamma}(s)+k_n(s)n_1(s)\right)+\sinh \theta n_2(s)$$ and $\tanh \theta=\dfrac{k_gk'_n+k_g^2\tau_1-k_n^2\tau_1-k_nk'_g}{(k_n\tau_2+k_g\tau_g)\sqrt{k_g^2-k_n^2}}(s)$.

For realize the calculations of the items (4) and (5) we use the Frenet-Serret type formulae of $\gamma$.  As the calculations are laborious and long we omit the details here.
\end{proof}

Following Proposition \ref{singht}, we define   the invariant \\

$\rho(s)=((-k_gk''_n-k_gk_n\tau_2^2-2k_gk'_g\tau_1-k_g^2\tau'_1-k_g^2\tau_g\tau_2+2k_nk'_n\tau_1+k_n^2\tau'_1-k_n^2k_g\tau_2+k''_gk_n-k_gk_n\tau_g^2)
    (k_n\tau_2+k_g\tau_g)+(k_gk'_n+k_g^2\tau_1-k_n^2\tau_1-k_nk'_g)(2k'_n\tau_2+k_n\tau_1\tau_g+k_n\tau'_2+2k'_g\tau_g+k_g\tau_1\tau_2+k_g\tau'_g))(s)$\\\\ of the curve $\gamma$. We will study the geometric meaning of this invariant.

     Motivated by the  calculations of this proposition we define a surface and its singular locus. Let $\gamma:I\rightarrow M$ be a  unit speed curve with $k_g(s)\neq0$ and $(k_n\tau_2+k_g\tau_g)(s)\neq0$, a surface $S_{\gamma}:I\times \mathbb{R}\rightarrow H^3(-1)$ is defined  by
$$S_{\gamma}(s,\theta)=\frac{\cosh \theta}{\sqrt{k_g^2(s)-k_n^2(s)}}\left(k_g(s)n_{\gamma}(s)+k_n(s)n_1(s)\right)+\sinh \theta n_2(s).$$
We call $S_{\gamma}$ a \emph{ hyperbolic surface} of $\gamma$. Since we assume  that $k_g^2(s)>k_n^2(s)$ for any $s\in I$, the hyperbolic surface exists. We now define $CH_{\gamma}=S_{\gamma}(s,\theta(s))$, where $\tanh\theta(s)=\dfrac{k_gk'_n+k_g^2\tau_1-k_n^2\tau_1-k_nk'_g}{(k_n\tau_2+k_g\tau_g)\sqrt{k_g^2-k_n^2}}(s)$. This is generically a curve. We call  $CH_{\gamma}$ a \emph{hyperbolic curve} of $\gamma$.  By  Theorem \ref{teo:classificacao} (1), this curve is the locus of the singular points of the hyperbolic surface of $\gamma$.

\begin{cor}\label{coro:discsuphiper}
 The hyperbolic surface of  $\gamma$ is the discriminant set $\mathcal{D}_{H_t^T}$  of the family of timelike tangential height functions  $H^{T}_t$.
\end{cor}
\begin{proof}
The proof follows from the definition of the discriminant set given in the Section \ref{sec:pre} and by  Proposition  \ref{singht} (2).
\end{proof}

    In the following proposition, we show that the family of timelike tangential height functions on a curve in $M$ is a versal deformation of an $A_k$-singularity, $k=2,3$, of its members.  Furthermore, we will study the geometric meaning of the  invariant $\rho$.  We write $\lambda_0(s)=\left(k_gk'_n+k_g^2\tau_1-k_n^2\tau_1-k_nk'_g\right)(s)$.

\begin{prop}\label{desdobramentohtT}
Let $\gamma:I\rightarrow M$ be a unit speed curve  with $k_g(s)\neq0$ and $(k_n\tau_2+k_g\tau_g)(s)\neq0$.
 \begin{itemize}
   \item [(a)] If $(h^{T}_t)_{v_0}$ has an $A_2$-singularity at $s_0$, then $H^{T}_t$ is a versal deformation of $(h^{T}_t)_{v_0}$.
   \item [(b)]  If $(h^{T}_t)_{v_0}$ has an $A_3$-singularity at $s_0$ and $\lambda_0(s_0)\neq0$ (which is a generic condition), then $H^{T}_t$ is a versal deformation of $(h^{T}_t)_{v_0}$.
 \end{itemize}

 %In other words, $H^{T}_t$ is generically  a versal deformation of $(h^{T}_t)_{v_0}$.
\end{prop}
\begin{proof}
The family of timelike tangential height functions  is given by $$H^{T}_t(s,v)=-v_0x'_0(s)+v_1x'_1(s)+v_2x'_2(s)+v_3x'_3(s),$$ where $v=(v_0,v_1,v_2,v_3)$, $t(s)=(x'_0(s),x'_1(s),x'_2(s),x'_3(s))$ and $v_0=\sqrt{1+v_1^2+v_2^2+v_3^2}$.

Thus $$\dfrac{\partial H^{T}_t}{\partial v_i}(s,v)=x'_i(s)-\dfrac{v_i}{v_0}x'_0(s),$$ for $i=1,2,3$. Therefore,  the $1$-jet of $\dfrac{\partial H^{T}_t}{\partial v_i}(s,v)$ at $s_0$ is given by $$x'_i(s_0)-\dfrac{v_i}{v_0}x'_0(s_0)+\left(x''_i(s_0)-\dfrac{v_i}{v_0}x''_0(s_0)\right)(s-s_0)$$

\noindent   and the $2$-jet of $\dfrac{\partial H^{T}_t}{\partial v_i}(s,v)$ at $s_0$ is given by $$x'_i(s_0)-\dfrac{v_i}{v_0}x'_0(s_0)+\left(x''_i(s_0)-\dfrac{v_i}{v_0}x''_0(s_0)\right)(s-s_0)+\dfrac{1}{2}\left(x'''_i(s_0)-
\dfrac{v_i}{v_0}x'''_0(s_0)\right)(s-s_0)^2.$$

We assume first that $(h^{T}_t)_v$ has an $A_2$-singularity at $s=s_0$. We show that the rank of the matrix

  $$B=\left(
                                                                                        \begin{array}{ccc}
x'_1(s_0)-\dfrac{v_1}{v_0}x'_0(s_0) &  x'_2(s_0)-\dfrac{v_2}{v_0}x'_0(s_0)  &  x'_3(s_0)-\dfrac{v_3}{v_0}x'_0(s_0)  \\
       x''_1(s_0)-\dfrac{v_1}{v_0}x''_0(s_0)  &  x''_2(s_0)-\dfrac{v_2}{v_0}x''_0(s_0)  & x''_3(s_0)-\dfrac{v_3}{v_0}x''_0(s_0)
                                                                                        \end{array}
                                                                                      \right)
.$$

\noindent is two.

We calculate the Gram-Schmidt matrix of  $\widetilde{B}={v_0}B$.  We denote the lines of $\widetilde{B}$ by
%$$\widetilde{B}={v_0}B=\left(
%                                                                                        \begin{array}{ccc}
%x'_1(s_0)v_0-x'_0(s_0)v_1 &  x'_2(s_0)v_0-x'_0(s_0)v_2  &  x'_3(s_0)v_0-x'_0(s_0)v_3  \\\\
 %      x''_1(s_0)v_0-x''_0(s_0)v_1  &  x''_2(s_0)v_0-x''_0(s_0)v_2  & x''_3(s_0)v_0-x''_0(s_0)v_3
   %                                                                                     \end{array}
  %                                                                                    \right).$$

 $$F=(x'_1(s_0)v_0-x'_0(s_0)v_1,x'_2(s_0)v_0-x'_0(s_0)v_2,x'_3(s_0)v_0-x'_0(s_0) v_3 ),$$
 $$G=(x''_1(s_0)v_0-x''_0(s_0)v_1, x''_2(s_0)v_0-x''_0(s_0)v_2, x''_3(s_0)v_0-x''_0(s_0)v_3).$$

 Since $\langle v,v\rangle=-1$, $\langle t(s),t(s)\rangle=1,$ $\langle t(s),v\rangle=0,$  $\langle t'(s),v\rangle=0$ and $\langle t'(s),t'(s)\rangle=k_g^2(s)-k_n^2(s)$, we have  the following Euclidean  inner product
$$F.F=v^2_0-{(x'_0)}^2,\,\,\,\,F.G=-x'_0x''_0\,\,\,\,\mbox{and}\,\,\,\,G.G=v^2_0(k_g^2(s)-k_n^2(s))-{(x''_0)}^2.$$

Therefore the Gram-Schmidt matrix of $\widetilde{B}$ is $$G_{\widetilde{B}}=\left(
                                                            \begin{array}{cc}
                                                              v^2_0-{(x'_0)}^2 & -x'_0x''_0 \\
                                                              -x'_0x''_0 & v^2_0(k_g^2(s)-k_n^2(s))-{(x''_0)}^2 \\
                                                            \end{array}
                                                          \right).
$$
 %Furthermore, $(h^{T}_t)_v$ has an $A_2$-singularity at $s=s_0$ if and only if $$v=\frac{\cosh %\theta_0}{\sqrt{k_g^2(s_0)-k_n^2(s)}}\left(k_g(s_0)n_{\gamma}(s_0)+k_n(s_0)n_1(s_0)\right)+\sinh \theta_0 n_2(s_0).$$
  By a Lorentzian motion of the curve, we can assume that $n_{\gamma}(s_0)=(1,0,0,0)$. In this case, we have $x'_0(s_0)=0$, $x''_0(s_0)=k_n(s_0)$ and $v_0=\dfrac{k_g(s_0)\cosh\theta_0}{\sqrt{k_g^2(s)-k_n^2(s)}}$. Thus the determinant  of $G_{\widetilde{B}}$ is $$v^2_0\left(k_g^2(s_0)-k_n^2(s_0)\right)\left(v^2_0-{(x'_0)}^2\right)-{v_0}^2{(x''_0)}^2=\dfrac{k_g^2(s_0)\cosh^2\theta_0}{{k_g^2(s_0)-k_n^2(s_0)}}\left(k_g^2(s_0)
  \cosh^2\theta_0-k_n^2(s_0)\right),$$ that is different from zero since $k_g^2(s_0)>k_n^2(s_0)$.  Thus the rank of the matrix B is equal to two and so the assertion (a) follows.

 We now assume that $(h^{T}_t)_v$ has an $A_3$-singularity at $s=s_0$. In this case, we show that the determinant of the $3\times3$ matrix
$$A=\left(
    \begin{array}{ccc}
      x'_1(s_0)-\dfrac{v_1}{v_0}x'_0(s_0) &  x'_2(s_0)-\dfrac{v_2}{v_0}x'_0(s_0)  &  x'_3(s_0)-\dfrac{v_3}{v_0}x'_0(s_0)  \\
       x''_1(s_0)-\dfrac{v_1}{v_0}x''_0(s_0)  &  x''_2(s_0)-\dfrac{v_2}{v_0}x''_0(s_0)  & x''_3(s_0)-\dfrac{v_3}{v_0}x''_0(s_0)  \\
       x'''_1(s_0)-\dfrac{v_1}{v_0}x'''_0(s_0) & x'''_2(s_0)-\dfrac{v_2}{v_0}x'''_0(s_0) & x'''_3(s_0)-\dfrac{v_3}{v_0}x'''_0(s_0) \\
    \end{array}
  \right)
$$
is nonzero. Denote  $$a=\left(
                                  \begin{array}{c}
                                    x'_0(s_0) \\
                                    x''_0(s_0) \\
                                    x'''_0(s_0) \\
                                  \end{array}
                                \right)
, b_i=\left(
   \begin{array}{c}
     x'_i(s_0)\\
     x''_i(s_0) \\
     x'''_i(s_0)\\
   \end{array}
 \right),
$$ for $i=1,2,3$. Then,  $$\det A=\dfrac{v_0}{v_0}\det (b_1\,\, b_2\,\,  b_3)
-\dfrac{v_1}{v_0}\det( a \,\, b_2 \,\, b_3) -\dfrac{v_2}{v_0}\det(  b_1\,\, a \,\, b_3 )-\dfrac{v_3}{v_0}\det ( b_1 \,\, b_2 \,\, a).
                                                               $$

On the other hand,

$$(\gamma'\wedge\gamma''\wedge\gamma''')(s_0)=(-\det (b_1\,\, b_2\,\,  b_3),-\det( a \,\, b_2 \,\, b_3),- \det(  b_1\,\, a \,\, b_3 ),-\det( b_1 \,\, b_2 \,\, a)).$$

Therefore,
\begin{align*}
\det A&=\left\langle\left(\dfrac{v_0}{v_0},\dfrac{v_1}{v_0},\dfrac{v_2}{v_0},\dfrac{v_3}{v_0}\right),(\gamma'\wedge\gamma''\wedge\gamma''')(s_0)\right\rangle\\
&=\dfrac{\cosh \theta_0 \left(k_gk'_n+k_g^2\tau_1-k_n^2\tau_1-k_nk'_g\right)^2}{v_0\sqrt{k_g^2-k_n^2}(k_n\tau_2+k_g\tau_g)}(s_0).
\end{align*}
 Therefore, if $(h^{T}_t)_{v_0}$ has an $A_3$-singularity  at $s_0$ and $\lambda_0(s_0)\neq0$, then $\det A\neq0$ and $H^{T}_t$ is a versal deformation of $(h^{T}_t)_{v_0}$. This completes the proof.

% When $k=2$, we require the rank of $B$ is equal to 2, where $B$ is the matrix  $$B=\left(
%                                                                                        \begin{array}{ccc}
%x'_1(s_0)-\dfrac{v_1}{v_0}x'_0(s_0) &  x'_2(s_0)-\dfrac{v_2}{v_0}x'_0(s_0)  &  x'_3(s_0)-\dfrac{v_3}{v_0}x'_0(s_0)  \\
%       x''_1(s_0)-\dfrac{v_1}{v_0}x''_0(s_0)  &  x''_2(s_0)-\dfrac{v_2}{v_0}x''_0(s_0)  & x''_3(s_0)-\dfrac{v_3}{v_0}x''_0(s_0)
%                                                                                        \end{array}
%                                                                                      \right)
%.$$
 %Since $B$ is the first and second lines of $A$, the rank of $B$ is two.
\end{proof}

By Proposition \ref{desdobramentohtT}, if  $(h^{T}_t)_{v_0}$ has an $A_3$-singularity at $s_0$ and  $\lambda_0(s_0)\neq0$  then $H^{T}_t$ is a versal deformation of $(h^{T}_t)_{v_0}$.  Now we investigate what is happening if  $\lambda_0(s_0)=0$.

%We assume in this section that $(h^{T}_t)_{v_0}$ and  $(h^{S}_t)_{v_0}$ have an $A_3$-singularity at $s_0$ and $\lambda_0(s_0)=0$.

First of all, we define a new deformation of $(h^{T}_t)_{v_0}$ and prove that it is a versal deformation. After, we use the Recognition Lemma for the cuspidal beaks or the cuspidal lips given in \cite{beaks}.

Using Proposition \ref{singht} with $\lambda_0(s_0)=0$, $(h^{T}_t)_{v_0}$ has an $A_3$-singularity at $s_0$ if and only if $\theta=0$, $v(s_0)=\dfrac{1}{\sqrt{k_g^2(s_0)-k_n^2(s_0)}}\left(k_gn_{\gamma}+k_nn_1\right)(s_0),$ $\rho(s_0)=0$ and $\rho'(s_0)\neq0$, where $\rho'(s_0)= \dfrac{-1}{\sqrt{k_g^2(s_0)-k_n^2(s_0)}}(-3\lambda_0'(s_0)\lambda_1(s_0)+\lambda_2(s_0))\neq0$, $\lambda_1(s)=(k'_n\tau_2+k_n\tau'_2+k'_g\tau_g+k_g\tau'_g)(s)$ and
$\lambda_2(s)=(k_gk'''_n+3k''_gk_g\tau_1+3k'_g\tau'_1k_g+k_g^2\tau''_1+k_g^2\tau_g\tau'_2-k_g^2\tau_g^2\tau_1-
k_n\tau_1\tau_2k_g^2+k_n\tau_1\tau_2k_g\tau_g-3k_nk''_n\tau_1-3k_nk'_n\tau'_1-k^2_n\tau''_1+k_nk'''_g+k_n^2\tau_1\tau_g^2+k_g^2\tau_1\tau_2^2-k_n^2\tau_1\tau_2^2+
k_n^2\tau_2\tau'_g-k_n^2\tau'_2\tau_g-k_g^2\tau'_g\tau_2+2\tau_1^2k'_nk_g-2\tau_1^2k'_gk_n)(s)$.

We now  define a deformation $\widetilde{H}:I\times H^3(-1)\times \mathbb{R}\rightarrow \mathbb{R}$ by $\widetilde{H}(s,v,u)=H^{T}_{t}(s,v)+u(s-s_0)^2=\langle t(s),v\rangle+u(s-s_0)^2.$ Here we consider the germ at $(s_0,v_0,0)$ represented by  $\widetilde{H}$.

\begin{prop}\label{prop:defnova}
If $(h^{T}_t)_{v_0}$  has an $A_3$-singularity at $s_0$ and $\lambda_0(s_0)=0$, then $\widetilde{H}$ is a  versal deformation of $(h^{T}_t)_{v_0}$.
\end{prop}
\begin{proof}
We have  $$\widetilde{H}(s,v,u)=H^{T}_t(s,v)+u(s-s_0)^2=-v_0x'_0(s)+v_1x'_1(s)+v_2x'_2(s)+v_3x'_3(s)+u(s-s_0)^2,$$ where $v=(v_0,v_1,v_2,v_3)$, $t(s)=(x'_0(s),x'_1(s),x'_2(s),x'_3(s))$ and $v_0=\sqrt{1+v_1^2+v_2^2+v_3^2}$.

Thus $$\dfrac{\partial \widetilde{H}}{\partial v_i}(s,v,0)=x'_i(s)-\dfrac{v_i}{v_0}x'_0(s),$$ for $i=1,2,3$. Therefore,  the $2$-jet of $\dfrac{\partial \widetilde{H}}{\partial v_i}(s,v,0)$ at $s_0$ is $$x'_i(s_0)-\dfrac{v_i}{v_0}x'_0(s_0)+\left(x''_i(s_0)-\dfrac{v_i}{v_0}x''_0(s_0)\right)(s-s_0)+\dfrac{1}{2}\left(x'''_i(s_0)-
\dfrac{v_i}{v_0}x'''_0(s_0)\right)(s-s_0)^2,$$
and the $2$-jet of  $\dfrac{\partial \widetilde{H}}{\partial u}(s,v,0)$ at $s_0$ is $(s-s_0)^2$.

 We assume that $(h^{T}_t)_v$ has an $A_3$-singularity at $s=s_0$. It is enough to show that \\

\noindent rank$\left(
    \begin{array}{cccc}
      x'_1(s_0)-\dfrac{v_1}{v_0}x'_0(s_0) &  x'_2(s_0)-\dfrac{v_2}{v_0}x'_0(s_0)  &  x'_3(s_0)-\dfrac{v_3}{v_0}x'_0(s_0)&0  \\
       x''_1(s_0)-\dfrac{v_1}{v_0}x''_0(s_0)  &  x''_2(s_0)-\dfrac{v_2}{v_0}x''_0(s_0)  & x''_3(s_0)-\dfrac{v_3}{v_0}x''_0(s_0)&0  \\
       x'''_1(s_0)-\dfrac{v_1}{v_0}x'''_0(s_0) & x'''_2(s_0)-\dfrac{v_2}{v_0}x'''_0(s_0) & x'''_3(s_0)-\dfrac{v_3}{v_0}x'''_0(s_0)&1 \\
    \end{array}
  \right)=$\\\\rank $=\left(
            \begin{array}{ccc}
             0& 0 & 1\\
              x'_1(s_0)-\dfrac{v_1}{v_0}x'_0(s_0) &  x''_1(s_0)-\dfrac{v_1}{v_0}x''_0(s_0) & 0 \\
               x'_2(s_0)-\dfrac{v_2}{v_0}x'_0(s_0)  & x''_2(s_0)-\dfrac{v_2}{v_0}x''_0(s_0) & 0 \\
               x'_3(s_0)-\dfrac{v_3}{v_0}x'_0(s_0) & x''_3(s_0)-\dfrac{v_3}{v_0}x''_0(s_0) & 0
              \\
            \end{array}
          \right)=3.
$
\vspace{0.5cm}
%Since $\langle t(s),t(s)\rangle=1$ and $\langle t'(s),t(s)\rangle=0$, we have $$-{x'_0}^2+{x'_1}^2+{x'_2}^2+{x'_3}^2=1,\,\,\,\,\,-x'_0x''_0+x'_1x''_1+x'_2x''_2+x'_3x''_3=0.$$

The rank of the last matrix has the same value as the rank of
$$\left(
            \begin{array}{ccc}
             1& 0 & 1\\
              x'_1(s_0)-\dfrac{v_1}{v_0}x'_0(s_0) &  x''_1(s_0)-\dfrac{v_1}{v_0}x''_0(s_0) & 0 \\
               x'_2(s_0)-\dfrac{v_2}{v_0}x'_0(s_0)  & x''_2(s_0)-\dfrac{v_2}{v_0}x''_0(s_0) & 0 \\
               x'_3(s_0)-\dfrac{v_3}{v_0}x'_0(s_0) & x''_3(s_0)-\dfrac{v_3}{v_0}x''_0(s_0) & 0
              \\
            \end{array}
          \right).$$

          Consider $$a(s_0)=\left(1, x'_1(s_0)-\dfrac{v_1}{v_0}x'_0(s_0),x'_2(s_0)-\dfrac{v_2}{v_0}x'_0(s_0),x'_3(s_0)-\dfrac{v_3}{v_0}x'_0(s_0)\right),$$ $$b(s_0)=\left(0,x''_1(s_0)-\dfrac{v_1}{v_0}x''_0(s_0), x''_2(s_0)-\dfrac{v_2}{v_0}x''_0(s_0),x''_3(s_0)-\dfrac{v_3}{v_0}x''_0(s_0) \right)$$ and $c(s_0)=(1,0,0,0)$. We can show that $a(s_0),$ $b(s_0)$, $c(s_0)$ are linearly independent. Indeed, if $a(s_0),$ $b(s_0)$, $c(s_0)$ are linearly dependent then $x'_1(s_0)=\dfrac{v_1}{v_0}x'_0(s_0)$, $x'_2(s_0)=\dfrac{v_2}{v_0}x'_0(s_0)$ and $x'_3(s_0)=\dfrac{v_3}{v_0}x'_0(s_0)$, that is, $t(s_0)$ and $v$ are parallel and so we have a  contradiction because $t$ is spacelike and $v$ is timelike.

          % Hence the rank of the above matrix is three. This completes the proof.
 \end{proof}

The cuspidal beaks is defined to a germ of surface diffeomorphic to $CBK = \{(x_1, x_2, x_3)|x_1 = v, x_2 = -2u_3 + v_2u, x_3 = 3u_4 - v_2u_2\}$. See the picture in \cite{beaks}.
Using Theorem \ref{teobruce1}, Propositions  \ref{desdobramentohtT} and \ref{prop:defnova}, we can obtain the diffeomorphism type  of the hyperbolic surface in the following theorem.

\begin{theo}\label{teo:classificacao}
Let $\gamma:I\rightarrow M$ be a  unit speed curve  with $k_g(s)\neq0$, $(k_n\tau_2+k_g\tau_g)(s)\neq0$ and $k_g^2(s)>k_n^2(s)$. Let  $S_{\gamma}$ be  the hyperbolic surface of $\gamma$. Then
\begin{itemize}
\item[(1)] $S_{\gamma}$ is singular at $(s_0, \theta_0)$ if and only if $$\tanh \theta_0=\dfrac{k_gk'_n+k_g^2\tau_1-k_n^2\tau_1-k_nk'_g}{\sqrt{k_g^2-k_n^2}(k_n\tau_2+k_g\tau_g)}(s_0).$$
    That is, the singular points of the hyperbolic surface are given by  $S_{\gamma}(s)=S_{\gamma}(s,\theta(s))$, where $\tanh \theta(s)$ satisfies the above equation.
\item[(2)]    The germ of  $S_{\gamma}$  at $(s_0, \theta_0)$ is  diffeomorphic to a cuspidal edge   if
     $$\tanh \theta_0=\dfrac{k_gk'_n+k_g^2\tau_1-k_n^2\tau_1-k_nk'_g}{(k_n\tau_2+k_g\tau_g)\sqrt{k_g^2-k_n^2}}(s_0)\,\,\mbox{and}\,\,\,\rho(s_0)\neq0.$$

\item[(3)] The germ of $S_{\gamma}$ at $(s_0, \theta_0)$ is diffeomorphic to a swallowtail  if
       $$\tanh \theta_0=\dfrac{k_gk'_n+k_g^2\tau_1-k_n^2\tau_1-k_nk'_g}{(k_n\tau_2+k_g\tau_g)\sqrt{k_g^2-k_n^2}}(s_0),\,\,
    \lambda_0(s_0)\neq0,\,\,\,\rho(s_0)=0\,\,\,\mbox {and}\,\,\, \rho'(s_0)\neq0.$$

    \item[(4)] The germ of $S_{\gamma}$ at $(s_0, \theta_0)$ is diffeomorphic to a cuspidal beaks if    $$
    \lambda_0(s_0)=0,\,\,\,\lambda_1(s_0)\neq0,\,\,\,\rho(s_0)=0\,\,\, \mbox{and}\,\,\,\rho'(s_0)\neq0.$$

\item[(5)] A cuspidal lips does not appear.
\end{itemize}
\end{theo}
\begin{proof}

We consider the hyperbolic surface $$S_{\gamma}(s,\theta)=\frac{\cosh \theta}{\sqrt{k_g^2(s)-k_n^2(s)}}\left(k_g(s)n_{\gamma}(s)+k_n(s)n_1(s)\right)+\sinh \theta n_2(s).$$ Then we have

\begin{align*}
\dfrac{\partial S_{\gamma}}{\partial s}(s,\theta)&=\Bigg(\dfrac{\cosh\theta(-k'_gk_n^2+k_gk_nk'_n+k_n\tau_1k_g^2-k_n^3\tau_1)+ \sinh\theta\tau_2(k_g^2 -k_n^2)\sqrt{k_g^2-k_n^2}}{(k_g^2-k_n^2)\sqrt{k_g^2-k_n^2}}\Bigg)(s)n_{\gamma}(s)\\
&+\Bigg(\dfrac{\cosh\theta(k_g^3\tau_1-k_g\tau_1k_n^2+k'_nk_g^2-k_nk_gk'_g)-\sinh\theta\tau_g(k_g^2-k_n^2)\sqrt{k_g^2-k_n^2}}{(k_g^2-k_n^2)\sqrt{k_g^2-k_n^2}}\Bigg)(s)n_1(s)\\
&+\Bigg(\dfrac{\cosh\theta(k_g\tau_2+k_n\tau_g)}{\sqrt{k_g^2-k_n^2}}\Bigg)(s)n_2(s)\,\,\,\,\mbox{and}
\end{align*}

$\dfrac{\partial S_{\gamma}}{\partial \theta}(s,\theta)=\dfrac{\sinh\theta k_g(s)}{\sqrt{k_g^2(s)-k_n^2(s)}}n_{\gamma}(s)+\dfrac{\sinh\theta k_n(s)}{\sqrt{k_g^2(s)-k_n^2(s)}}n_1(s)+\cosh\theta n_2(s)$.\\\\

Therefore,  the vectors  $\left\{\dfrac{\partial S_{\gamma}}{\partial s}(s_0,\theta_0), \dfrac{\partial S_{\gamma}}{\partial \theta}(s_0,\theta_0)\right\}$ are  linearly dependent if and only if $\tanh\theta_0=\dfrac{k_gk'_n+k_g^2\tau_1-k_n^2\tau_1-k_nk'_g}{(k_n\tau_2+k_g\tau_g)\sqrt{k_g^2-k_n^2}}(s_0)$ and thus assertion (1) holds.

% By Proposition \ref{singht}, \emph{(2)}, the discriminant set $\mathcal{D}_{H_t^T}$, of the family of timelike tangential height functions $H^{T}_t$ of $\gamma$, is the hyperbolic surface $S_{\gamma}$. The singularities of the discriminant set  are corresponding to the points of Proposition \ref{singht}, \emph{(3)}, so that  assertion \emph{(1)} holds.

 By Corollary \ref{coro:discsuphiper}, the discriminant set $\mathcal{D}_{H_t^T}$ of the family of timelike tangential height functions $H^{T}_t$ of $\gamma$ is the hyperbolic surface $S_{\gamma}$. It also follows from assertions (4) and (5) of Proposition \ref{singht} that $(h^{T}_t)_{v_0}$ has an  $A_2$-singularity (respectively, an $A_3$-singularity) at $s=s_0$ if and only if  $$\tanh \theta_0=\dfrac{k_gk'_n+k_g^2\tau_1-k_n^2\tau_1-k_nk'_g}{(k_n\tau_2+k_g\tau_g)\sqrt{k_g^2-k_n^2}}(s_0) \,\,\,\mbox{and}\,\,\,\rho(s_0)\neq0$$ (respectively, $\tanh \theta_0=\dfrac{k_gk'_n+k_g^2\tau_1-k_n^2\tau_1-k_nk'_g}{(k_n\tau_2+k_g\tau_g)\sqrt{k_g^2-k_n^2}}(s_0)$, $\rho(s_0)=0$ and $ \rho'(s_0)\neq0$). Therefore, by  Proposition \ref{desdobramentohtT}, we have  assertions (2) and  (3).

  By Proposition 7.5  in \cite{beaks} and  by the previous  Proposition \ref{prop:defnova}, $H^{T}_{t}$ is a Morse family of hypersurfaces.

 We now calculate $\varphi=(\partial^2 H^{T}_{t}/{\partial s^2})|\mathcal{D}_{H^{T}_{t}}.$ Then we have
\begin{align*}
\dfrac{\partial^2 H^{T}_{t}}{{\partial s^2}}(s,\theta)=&\left\langle t''(s),\frac{\cosh \theta}{\sqrt{k_g^2(s)-k_n^2(s)}}\left(k_g(s)n_{\gamma}(s)+k_n(s)n_1(s)\right)+\sinh \theta n_2(s)\right\rangle\\
=&\frac{-\cosh \theta}{\sqrt{k_g^2(s)-k_n^2(s)}}(k_gk'_n+k_g^2\tau_1-k_n^2\tau_1-k_nk'_g)(s)+ \sinh\theta (k_n\tau_2+k_g\tau_g)(s).
\end{align*}

The Hessian matrix of $\varphi(s,\theta)=\dfrac{-\cosh \theta}{\sqrt{k_g^2(s)-k_n^2(s)}}(k_gk'_n+k_g^2\tau_1-k_n^2\tau_1-k_nk'_g)(s)+ \sinh\theta (k_n\tau_2+k_g\tau_g)(s)$ is
$$\mbox{Hess}(\varphi)(s_0,0)=\left(
    \begin{array}{cc}
      \dfrac{\partial^2\varphi}{\partial s^2}(s_0,0) & \lambda_1(s_0)  \\
        \lambda_1(s_0)  & 0 \\
    \end{array}
  \right).
$$
Since $\lambda_1(s_0)\neq0$, we have $\det \mbox{Hess}(\varphi)(s_0,0)\neq 0$. By
Lemma 7.7 in \cite{beaks},  $H^{T}_{t}$ is $P$-$\mathcal{K}$-equivalent to $t^4\pm v_1^2t^2+v_2t +v^3$ (the notion of generating families, Legendrian equivalence and $P$-$\mathcal{K}$-equivalent are given in \cite{beaks} page 30). The singular set of $\mathcal{D}_{H^{T}_{t}}$ is given by $\varphi(s,\theta)=0$. Therefore it consists of two curves that transversally intersect at $(s_0,0)$. So the normal form is  $t^4- v_1^2t^2+v_2t +v^3$ and   the surface is diffeomorphic to  the cuspidal beaks and we have  assertions (4) and (5).
\end{proof}

We have three types of models of surfaces in $M$, which are given by intersections of $M$ with hyperplanes in $\mathbb{R}^4_1$. We call a surface $M\cap HP(v,c)$  a \emph{timelike slice} if $v$ is spacelike, a \emph{spacelike slice} if $v$ is timelike or a \emph{lightlike slice} if $v$ is lightlike.

In the  following proposition we relate the curve $\gamma$ of the hyperbolic surface with the invariant $\rho$ and a slice surface. In this case, the singular locus of the  hyperbolic surface of $\gamma$ is a point.

\begin{prop}\label{pro:slice1}
Let $\gamma:I\rightarrow M$ be a unit speed curve such that $k_g(s)\neq0$, $(k_n\tau_2+k_g\tau_g)(s)\neq0$ and  $k_g^2(s)>k_n^2(s)$ for any $s\in I$.  Let  $S_{\gamma}(s,\theta(s))$ be the singular points of the hyperbolic surface of $\gamma$. Then the following conditions are equivalent:
\begin{itemize}
\item[(1)] $S_{\gamma}(s,\theta(s))$ is a constant timelike vector;
\item[(2)] $\rho(s)\equiv 0$;
\item[(3)] there exists a timelike vector $v$ and a real number $c$ such that $Im(\gamma)\subset M\cap HP(v,c)$.
\end{itemize}
\end{prop}
\begin{proof}
By definition

\noindent$S_{\gamma}(s,\theta(s))= \dfrac{\cosh \theta(s)}{\sqrt{(k_g^2-k_n^2)(s)}}\left((k_gn_{\gamma})(s)+(k_nn_1)(s)+\dfrac{(k_gk'_n+k_g^2\tau_1-k_n^2\tau_1-k_nk'_g)(s)}{(k_n\tau_2+k_g\tau_g)(s)}n_2(s)\right).$

Thus,
\begin{align*}
&\dfrac{dS_{\gamma}(s,\theta(s))}{ds}=\\&\left(\dfrac{\cosh \theta(s)}{\sqrt{k_g^2(s)-k_n^2(s)}}\right)'\left(k_g(s)n_{\gamma}(s)+k_n(s)n_1(s)+\dfrac{(k_gk'_n+k_g^2\tau_1-k_n^2\tau_1-k_nk'_g)}
{(k_n\tau_2+k_g\tau_g)}(s)n_2(s)\right)\\
&+\left(\dfrac{\cosh
\theta(s)}{\sqrt{k_g^2(s)-k_n^2(s)}}\right)\left(k_g(s)n_{\gamma}(s)+k_n(s)n_1(s)+\dfrac{(k_gk'_n+k_g^2\tau_1-k_n^2\tau_1-k_nk'_g)}
{(k_n\tau_2+k_g\tau_g)}(s)n_2(s)\right)'.
\end{align*}
Furthermore,
$$\theta'(s)=\dfrac{X(s)}
{\sqrt{(k_g^2-k_n^2)(s)}((k_g^2-k_n^2)(k_n\tau_2+k_g\tau_g)^2-(k_gk'_n+k_g^2\tau_1-k_n^2\tau_1-k_nk'_g)^2)(s)},$$

\noindent where $X(s)=(k_gk_n''+2k_gk_g'\tau1
+k_g^2\tau_1'-2k_nk_n'\tau_1-k_n^2\tau_1'-k_nk_g'')(k_g^2-k_n^2)(k_n\tau_2+k_g\tau_g)-(k_gk'_n+k_g^2\tau_1-k_n^2\tau_1-k_nk'_g)((k_gk_g'-k_nk_n')(k_n\tau_2+k_g\tau_g)
+(k_g^2-k_n^2)(k_n'\tau_2+k_n\tau_2'+k_g'\tau_g+k_g\tau'_g))(s)$.

Using the  Frenet-Serret type formulae, replacing  $\theta'(s)$ in the previous expression of the derivative and making some calculations, we have that

\begin{align*}
&\dfrac{dS_{\gamma}(s,\theta(s))}{ds}=\\&\dfrac{-\cosh \theta (an_{\gamma}+bn_1+cn_2)\rho}{\sqrt{k_g^2-k_n^2}(k_n\tau_2+k_g\tau_g)((k_g^2-k_n^2)(k_n\tau_2+k_g\tau_g)^2-(k_gk'_n+k_g^2\tau_1-k_n^2\tau_1-k_nk'_g)^2)}(s),
\end{align*}

\noindent where $a(s)=k_g(k_n'k_g+k_g^2\tau_1-k_n^2\tau_1-k_g'k_n)(s)$, $b(s)=k_n(k_n'k_g+k_g^2\tau_1-k_n^2\tau_1-k_g'k_n)(s)$, $c(s)=(k_g^2-k_n^2)(k_n\tau_2+k_g\tau_g)(s)$ and $\rho(s)$ is the  invariant.\vspace{0.5cm}

Therefore, \noindent$\dfrac{dS_{\gamma}}{ds}\equiv0$ if and only if $\rho(s)\equiv0$. This means that the statements  (1) and (2) are equivalent. We now assume that the statement (1) holds, then we have

\vspace{0.5cm}

$\langle \gamma(s),S_{\gamma}(s,\theta(s))\rangle=\\
\dfrac{\cosh \theta}{\sqrt{k_g^2-k_n^2}}\left(k_g\langle\gamma,n_{\gamma}\rangle+k_n\langle\gamma,n_1\rangle+\dfrac{(k_gk_n'+k_g^2\tau_1-k_n^2\tau_1-k_nk_g')}{k_n\tau_2+k_g\tau_g}
\langle\gamma,n_2\rangle\right)(s).$

\vspace{0.5cm}

Let $g(s)=\langle \gamma(s),S_{\gamma}(s,\theta(s))\rangle$, deriving, using the Frenet-Serret type formulae  and  making long calculations we show that
%\vspace{-0.3cm}
 $$g'(s)=g_1(s)\langle \gamma(s), n_{\gamma}(s)\rangle+g_2(s) \langle \gamma(s),n_1(s)\rangle+g_3(s)\langle\gamma(s),n_2(s)\rangle,$$
 where
$
g_1(s)=\dfrac{A(s)\cosh \theta(s) }{D(s)},\,\,\,\,g_2(s)=\dfrac{B(s)\cosh \theta(s) }{D(s)}\,\,\,\,\hbox{and}\,\,\,\,g_3(s)=\dfrac{C(s)\cosh \theta(s)}{D_1(s)}
$
\noindent with
\begin{align*}
A(s)=&\Bigg(k_g(k_gk_n'+k_g^2\tau_1-k_n^2\tau_1-k_nk_g')\Big[(k_gk_n''+2k_gk_g'\tau_1+k_g^2\tau_1'-2k_nk_n'\tau_1-k_n^2\tau_1'-k_nk_g'')\\
&(k_g^2-k_n^2)(k_n\tau_2+k_g\tau_g)-(k_gk_n'+k_g^2\tau_1-k_n^2\tau_1-k_nk_g')\Big((k_gk_g'-k_nk_n')(k_n\tau_2+k_g\tau_g)\\
&+(k_g^2-k_n^2)(k_n'\tau_2+k_n\tau_2'+k_g'\tau_g+k_g\tau_g')\Big)\Big]-k_g(k_gk_g'-k_nk_n')(k_n\tau_2+k_g\tau_g)^3(k_g^2-k_n^2)\\
&+k_g(k_gk_g'-k_nk_n')(k_n\tau_2+k_g\tau_g)(k_gk_n'+k_g^2\tau_1-k_n^2\tau_1-k_nk_g')^2+\Big((k_n\tau_2+k_g\tau_g)(k_g'+\\
&k_n\tau_1)
+\tau_2k_gk_n'+\tau_2k_g^2\tau_1-\tau_2k_n^2
\tau_1-\tau_2k_nk_g'\Big)(k_g^2-k_n^2)^2(k_n\tau_2+k_g\tau_g)^2-\Big((k_g'+k_n\tau_1)\\
&(k_n\tau_2+k_g\tau_g)
+\tau_2k_gk_n'+\tau_2k_g^2\tau_1-\tau_2k_n^2\tau_1
-\tau_2k_nk_g'\Big)(k_g^2-k_n^2)(k_gk_n'+k_g^2\tau_1-k_n^2\tau_1-\\
&k_nk_g')^2\Bigg)(s),
\end{align*}
\begin{align*}
D(s)=&\Bigg((k_n\tau_2+k_g\tau_g)\sqrt{(k_g^2-k_n^2)^3}\Big((k_g^2-k_n^2)(k_n\tau_2+k_g\tau_g)^2-(k_gk_n'+k_g^2\tau_1-k_n^2\tau_1-k_nk_g')^2\Big)\Bigg)(s),
\end{align*}
\begin{align*}
B(s)=&\Bigg(k_n(k_gk_n'+k_g^2\tau_1-k_n^2\tau_1-k_nk_g')\Big[(k_gk_n''+2k_gk_g'\tau_1+k_g^2\tau_1'-2k_nk_n'\tau_1-k_n^2\tau_1'-k_nk_g'')\\
&(k_g^2-k_n^2)(k_n\tau_2+k_g\tau_g)-(k_gk_n'+k_g^2\tau_1-k_n^2\tau_1-k_nk_g')\Big((k_gk_g'-k_nk_n')(k_n\tau_2+k_g\tau_g)\\
&+(k_g^2-k_n^2)(k_n'\tau_2+k_n\tau_2'+k_g'\tau_g+k_g\tau_g')\Big)\Big]-k_n(k_gk_g'-k_nk_n')(k_n\tau_2+k_g\tau_g)^3(k_g^2-k_n^2)\\
&+k_n(k_gk_g'-k_nk_n')(k_n\tau_2+k_g\tau_g)(k_gk_n'+k_g^2\tau_1-k_n^2\tau_1-k_nk_g')^2+(k_gk_n\tau_1\tau_2
+k_nk_n'\tau_2\\
&+\tau_gk_n^2\tau_1+\tau_gk_nk_g')
(k_g^2-k_n^2)^2(k_n\tau_2+k_g\tau_g)-(k_gk_n\tau_1\tau_2+k_nk_n'\tau_2+\tau_gk_n^2\tau_1+\tau_gk_nk_g')
\\
&(k_g^2-k_n^2)
(k_gk_n'+k_g^2\tau_1-k_n^2\tau_1-k_nk_g')^2\Bigg)(s),
\end{align*}

\begin{align*}
C(s)=&\Bigg(-(k_n\tau_2+k_g\tau_g)(k_gk_n'+k_g^2\tau_1-k_n^2\tau_1-k_nk_g')^2(k_g\tau_2+k_n\tau_g)-(k_gk_n'+k_g^2\tau_1-k_n^2\tau_1\\
&-k_nk_g')
(k_n\tau_2+k_g\tau_g)^2(k_gk_g'-k_nk_n')+(k_gk_n'+k_g^2\tau_1-k_n^2\tau_1-k_nk_g')(k_n'\tau_2+k_n\tau_1\tau_g+\\
&k_g'\tau_g+
k_g\tau_1\tau_2)(k_g^2-k_n^2)(k_n\tau_2+k_g\tau_g)\Bigg)(s)\,\,\,\,\,\,\,\hbox{and}\\
D_1(s)=&\Bigg(\sqrt{k_g^2-k_n^2}(k_n\tau_2+k_g\tau_g)\Big((k_n^2-k_g^2)(k_n\tau_2+k_g\tau_g)^2-(k_gk_n'+k_g^2\tau_1-k_n^2\tau_1-k_nk_g')^2\Big)\Bigg)(s).
\end{align*}

Furthermore, reorganizing the calculations in $A(s)$, $B(s)$ and $C(s)$,  we show that $A(s)=B(s)=C(s)=0$  for all  $s\in I$ and thus  $g_i(s)=0$, $i=1,2,3$, for all $s\in I$, ( i.e.,   $g'(s)=0$ for all $s\in I$), so  that $g$ is  constant and the statement (3) follows.
For the converse, we assume that  $\langle\gamma(s),v\rangle=c$ for a constant vector $v$ and a real number $c$, thus  $\langle\gamma'(s),v\rangle=0$, that is, $(h_t^T)_v(s)=0$ for all $s$. By this way, we have
${({h^{T}_t})}_{v}(s)={({h^{T}_t})}_{v}'(s)={({h^{T}_t})}_{v}''(s)={({h^{T}_t})}_{v}^{'''}(s)=0$ for all $s$. By Proposition \ref{singht}, we have that  $v=S_{\gamma}(s,\theta(s))$ and $\rho(s)=0$ for all $s$. So (1) follows.
\end{proof}

In the Proposition \ref{pro:slice1} the invariant $\rho\equiv0$ means that the curve $\gamma$ is part of a spacelike slice surface. For the next result we assume that $\rho\not\equiv0$, that is $\gamma$ is not part of any spacelike slice surface $M\cap HP(v_0,c)$.

%\begin{defn}\label{defcontact}
%Let $F:\mathbb{R}^{4}_1\rightarrow \mathbb{R}$ (respectively, $F\mid_{M}:M\rightarrow \mathbb{R}$) be a submersion and $\gamma:I\rightarrow M$ be a regular curve. We say that $\gamma$ and $F^{-1}(0)$ (respectively $F^{-1}(0)\cap M$) have contact of order $k$ at $s_0$ if the function  $f(s)=F\circ\gamma(s)$ satisfies $f(s_0)=f'(s_0)=\cdots=f^{(k)}(s_0)=0$ and $f^{(k+1)}(s_0)\neq0$, i.e., $f$ has $A_k$-type singularity at $s_0$.
%\end{defn}
We now consider the hyperbolic curve $CH_{\gamma}$ of $\gamma$, which was defined in Section \ref{sectionht}. We define $C(2,3,4)=\{(t^2,t^3,t^4)\,\mid\,t\in\mathbb{R}\}$, which is called a $(2,3,4)$-\emph{cusp}. We have the following result.
\begin{prop}\label{prop:contato1}
Let $\gamma:I\rightarrow M$ be a unit speed curve  such that $k_g(s)\neq0$, $(k_n\tau_2+k_g\tau_g)(s)\neq0$ and $k_g^2(s)>k_n^2(s)$ for any $s\in I$. Let $v_0=S_{\gamma}(s_0,\theta_0)$ and $c=\langle \gamma(s_0),v_0\rangle$. Then we have the following:
\begin{itemize}
\item[(1)] $\gamma$ and the spacelike slice surface $M\cap HP(v_0,c)$ have contact of at least order 3 at $s_0$ if and only if ${({h^{T}_t})}_{v_0}$ has $A_{k}$-singularity at $s_0$, $k\geq 2$. Furthermore, if  $\gamma$ and the spacelike slice surface $M\cap HP(v_0,c)$ have contact of  order exactly 3 at $s_0$, then the  hyperbolic curve $CH_{\gamma}$ of $\gamma$ is, at  $s_0$, locally diffeomorphic to a line.
\item[(2)] $\gamma$ and the spacelike slice surface $M\cap HP(v_0,c)$  have contact of order 4 at $s_0$ if and only if $ {({h^{T}_t})}_{v_0}$ has $A_{3}$-singularity at $s_0$. In this case, if  $\lambda_0(s_0)\neq0$ then, the  hyperbolic curve $CH_{\gamma}$ of $\gamma$ is, at  $s_0$, locally diffeomorphic to the $(2,3,4)$-\emph{cusp} $C(2,3,4)$.
\end{itemize}
\end{prop}

\begin{proof}
  Consider $v_0=S_{\gamma}(s_0,\theta_0)$ and $c=\langle \gamma(s_0),v_0\rangle$. Let $D_{v_0}:M\rightarrow \mathbb{R}$ be a function defined by $D_{v_0}(x)=\langle x,v_0\rangle-c$. Thus, we have that $D_{v_0}^{-1}(0)=M\cap HP(v_0,c)$, which is a spacelike slice surface. Furthermore, $D_{v_0}^{-1}(0)$ and $\gamma$ have contact of at least order 3 at $s_0$ if and only if  the function $g(s)=D_{v_0}\circ \gamma(s)=\langle\gamma(s_0),v_0\rangle-c$ satisfies $g(s_0)=g'(s_0)=g''(s_0)=g'''(s_0)=0$. These conditions are equivalent to  $g(s_0)={({h^{T}_t})}_{v}(s)={({h^{T}_t})}_{v}'(s)={({h^{T}_t})}_{v}''(s)=0$. By Proposition \ref{singht}, they are equivalent to the condition that

 $$v_0=\dfrac{\cosh \theta_0}{\sqrt{k_g^2(s_0)-k_n^2(s_0)}}\left(k_g(s_0)n_{\gamma}(s_0)+k_n(s_0)n_1(s_0)\right)+\sinh \theta_0 n_2(s_0),$$  $\tanh \theta_0=\dfrac{k_gk'_n+k_g^2\tau_1-k_n^2\tau_1-k_nk'_g}{(k_n\tau_2+k_g\tau_g)\sqrt{k_g^2-k_n^2}}(s_0)$.
  If  $\gamma$ and the spacelike slice surface $M\cap HP(v_0,c)$  have contact of  order 3 at $s_0$, then we have that $$v_0=\dfrac{\cosh \theta_0}{\sqrt{k_g^2(s_0)-k_n^2(s_0)}}\left(k_g(s_0)n_{\gamma}(s_0)+k_n(s_0)n_1(s_0)\right)+\sinh \theta_0 n_2(s_0),$$  $\tanh \theta_0=\dfrac{k_gk'_n+k_g^2\tau_1-k_n^2\tau_1-k_nk'_g}{(k_n\tau_2+k_g\tau_g)\sqrt{k_g^2-k_n^2}}(s_0)$ and $\rho(s_0)\neq0.$ Furthermore, by Theorem \ref{teo:classificacao},  the germ of the image of the hyperbolic surface $S_{\gamma}$  at $(s_0, \theta_0)$ is locally diffeomorphic to the cuspidal edge. Since the locus of the singularities of cuspidal edge  is locally diffeomorphic to a line, the assertion (1) holds.

  The first part of (2) follows from  assertions (4) and (5) of Proposition \ref{singht}. For the second part, if  $\gamma$ and the spacelike slice surface $M\cap HP(v_0,c)$  have contact of  order 4 at $s_0$, then we have that $$v_0=\dfrac{\cosh \theta_0}{\sqrt{k_g^2(s_0)-k_n^2(s_0)}}\left(k_g(s_0)n_{\gamma}(s_0)+k_n(s_0)n_1(s_0)\right)+\sinh \theta_0 n_2(s_0),$$ $\tanh \theta_0=\dfrac{k_gk'_n+k_g^2\tau_1-k_n^2\tau_1-k_nk'_g}{(k_n\tau_2+k_g\tau_g)\sqrt{k_g^2-k_n^2}}(s_0)$, $\rho(s_0)=0$ and $\rho'(s_0)\neq0.$ Furthermore we have the assumption that $\lambda_0(s_0)\neq0$. By Theorem \ref{teo:classificacao}, the germ of the image of the  hyperbolic surface $S_{\gamma}$  at $(s_0, \theta_0)$ is locally diffeomorphic to the swallowtail surface. Since the locus of singularities of swallowtail surface is locally diffeomorphic to the $C(2,3,4)$, the assertion (2) holds.

\end{proof}

\section{Examples}\label{sec:exh3}
In this section, we consider two examples of curves on spacelike hypersurface $M$ in $\mathbb{R}^4_1$. One of them is $M=\mathbb{R}^3$ another is  $M=H^3(-1)$, which is the hyperbolic space.

\begin{ex}\rm\label{ex:1}
We consider $M=\mathbb{R}^3=\{x=(x_0,x_1,x_2,x_3)\in\mathbb{R}^4_1\mid x_0=0\}$. For $\gamma:I\rightarrow \mathbb{R}^3$, we have  $n_{\gamma}=e_0$, $t(s)=\gamma'(s)$, $n_1(s)=n(s)$ and $n_2(s)=b(s)$. Here $\{t, n, b\}$ is the ordinary Frenet frame. In this case, $k_n=\tau_1=\tau_2=0$, $k_g=k$ and $\tau_g=\tau$. The Frenet-Serret type formulae are the original Frenet-Serret formulae (see \cite{bruce}):
$$
\left\{
  \begin{aligned}
     e_{0}'(s) &=0, \\
     t'(s) & =k(s)\,n(s),\\
 n'(s) &= -k(s)\,t(s)+\tau(s)\,b(s),\\
  b'(s) &= -\tau(s)\,n(s).
  \end{aligned}
\right.$$

The hyperbolic surface of  $\gamma$ in $H^{3}(-1)\subset \mathbb{R}^4_1$ is given by  $$S_{\gamma}(s,\theta)=\cosh\theta e_0+\sinh\theta b(s)$$ and the hyperbolic curve of $\gamma$ is given by $CH_{\gamma}(s)=e_0$, which is a  constant point.
\end{ex}

\begin{ex}\rm\label{ex:2}
We consider $M=H^3(-1)$. For $\gamma:I\rightarrow H^3(-1)$, we have $n_{\gamma}(s)=\gamma(s)$, $t(s)=\gamma'(s)$, $n_1(s)$ and $n_2(s)$. Here $\{\gamma, t,n_1,n_2\}$ is the pseudo orthonormal frame. In this case, $k_n(s)=1$, $\tau_1(s)=\tau_2(s)=0$, $k_g(s)=k_h(s)$ and $\tau_g(s)=\tau_h(s)$.
$$
\left\{
  \begin{aligned}
     \gamma'(s) &=t(s), \\
     t'(s) & =\gamma(s)+k_h(s)\,n_1(s),\\
 n_1'(s) &= -k_h(s)\,t(s)+\tau_h(s)\,n_2(s),\\
  n_2'(s) &= -\tau_h(s)\,n_1(s).
  \end{aligned}
\right.$$
Therefore, for $k_h^2(s)>1$ the hyperbolic surface of $\gamma$ is given by
 $$S_{\gamma}(s,\theta)=\dfrac{\cosh \theta }{\sqrt{k_h^2(s)-1}}(k_h(s)\gamma(s)+n_1(s))+\sinh \theta n_2(s).$$
 Then the hyperbolic surface is precisely the hyperbolic focal surface of  $\gamma$ given in \cite{izumyiafs}.
\end{ex}

\section{ Spacelike tangential height functions}\label{sectionhs}

In this section we introduce the family of spacelike tangential height functions on a curve in a spacelike hypersurface $M$. Furthermore, we define and study the de Sitter surface which is  given by the discriminant set of this family. The arguments and results here are analogous to those of Section \ref{sectionht}, so that we do not present the detailed arguments in this section.

We define a family of  functions on a curve, $\gamma:I\rightarrow M\subset\mathbb{R}^4_1$ as follows:
$$H^{S}_t:I\times S^3_1\rightarrow \mathbb{R};\,\,\,\,(s,v)\mapsto\langle t(s),v\rangle.$$
We call  $H^{S}_t$ the \emph{family of spacelike tangential height functions of $\gamma$}. We denote $({h^{S}_t})_{v}(s)=H^{S}_t(s,v)$  for any fixed $v\in S^3_1$. The family  $H^{S}_t$ measures the contact of the curve $t$ with timelike hyperplanes in $\mathbb{R}^4_1$. Generically this contact can be of order $k$, $k=1,2,3$.

The conditions that  characterize the $A_k$-singularities, $k=1,2,3$  can be obtained in Proposition \ref{singhs}.

%Observe that by the proof of (2) in the below proposition, for the points of the curve where  $k_n^2(s)\leq k_g^2(s)$ the spacelike tangential height function ${h^{S}_t}$ has no %
%$A_{k}$-singularity  at $s$. In other words, if $\gamma$ is a curve where this inequality holds, for any  $s\in I$, then the discriminant set of $H^{S}_t$ on $\gamma$ is not defined.
 We assume that  $k_n^2(s)> k_g^2(s)$ for $s\in I$. Furthermore, in order to avoid  more complicated situations we assume that $(k_n\tau_2+k_g\tau_g)(s)\neq0$ for any $s\in I$.

\begin{prop}\label{singhs}
Let $\gamma:I\rightarrow M$ be a unit speed curve such that $k_g(s)\neq0$, $(k_n\tau_2+k_g\tau_g)(s)\neq0$ and $k_n^2(s)> k_g^2(s)$. Then, we have the following:
\begin{itemize}
\item[(1)] ${({h^{S}_t})}_{v}(s)=0$ if and only if  there exist $\mu$, $\lambda$, $\eta \in\mathbb{R}$ such that $-\mu^2+\lambda^2+\eta^2=1$ and  $v=\mu n_{\gamma}(s)+\lambda n_{1}(s)+\eta n_2(s)$.

\item[(2)] ${({h^{S}_t})}_{v}(s)={({h^{S}_t})}_{v}'(s)=0$ if and only if there exists $\theta\in\mathbb{R}$ such that $$v=\frac{\cos \theta}{\sqrt{k_n^2(s)-k_g^2(s)}}\left(k_g(s)n_{\gamma}(s)+k_n(s)n_1(s)\right)+\sin\theta n_2(s).$$

\item[(3)] ${({h^{S}_t})}_{v}(s)={({h^{S}_t})}_{v}'(s)={({h^{S}_t})}_{v}''(s)=0$ if and only if $$v=\frac{\cos \theta}{\sqrt{k_n^2(s)-k_g^2(s)}}\left(k_g(s)n_{\gamma}(s)+k_n(s)n_1(s)\right)+\sin \theta n_2(s)$$ and $\tan \theta=\dfrac{k_gk'_n+k_g^2\tau_1-k_n^2\tau_1-k_nk'_g}{(k_n\tau_2+k_g\tau_g)\sqrt{k_n^2-k_g^2}}(s)$.

\item[(4)]  ${({h^{S}_t})}_{v}(s)={({h^{S}_t})}_{v}'(s)={({h^{S}_t})}_{v}''(s)={({h^{S}_t})}_{v}^{'''}(s)=0$ if and only if $$v=\frac{\cos \theta}{\sqrt{k_n^2(s)-k_g^2(s)}}\left(k_g(s)n_{\gamma}(s)+k_n(s)n_1(s)\right)+\sin \theta n_2(s),$$ $\tan \theta=\dfrac{k_gk'_n+k_g^2\tau_1-k_n^2\tau_1-k_nk'_g}{(k_n\tau_2+k_g\tau_g)\sqrt{k_n^2-k_g^2}}(s)$ and $\rho(s)=0$, where

    $\rho(s)=\Bigg((-k_gk''_n-k_gk_n\tau_2^2-2k_gk'_g\tau_1-k_g^2\tau'_1-k_g^2\tau_g\tau_2+2k_nk'_n\tau_1+k_n^2\tau'_1-k_n^2k_g\tau_2+k''_gk_n-k_gk_n\tau_g^2)
    (k_n\tau_2+k_g\tau_g)+(k_gk'_n+k_g^2\tau_1-k_n^2\tau_1-k_nk'_g)(2k'_n\tau_2+k_n\tau_1\tau_g+k_n\tau'_2+2k'_g\tau_g+k_g\tau_1\tau_2+k_g\tau'_g)\Bigg)(s)$.

    \item[(5)]  ${({h^{S}_t})}_{v}(s)={({h^{S}_t})}_{v}'(s)={({h^{S}_t})}_{v}''(s)={({h^{S}_t})}_{v}^{'''}(s)={({h^{S}_t})}_{v}^{(4)}(s)=0$ if and only if $$v=\frac{\cos \theta}{\sqrt{k_n^2(s)-k_g^2(s)}}\left(k_g(s)n_{\gamma}(s)+k_n(s)n_1(s)\right)+\sin \theta n_2(s),$$ $\tan \theta=\dfrac{k_gk'_n+k_g^2\tau_1-k_n^2\tau_1-k_nk'_g}{(k_n\tau_2+k_g\tau_g)\sqrt{k_n^2-k_g^2}}(s)$ and  $\rho(s)=\rho'(s)=0$.
\end{itemize}
\end{prop}
%\begin{proof}
%The proof is similar to that of Proposition \ref{singht}.
%\emph{(1)} We have ${({h^{S}_t})}_{v}(s)=0$ if and only if $\langle t(s),v\rangle=0$. This is equivalent to
%$v=\mu n_{\gamma}(s)+\lambda n_{1}(s)+\eta n_2(s)$, where $\mu$, $\lambda$, $\eta \in\mathbb{R}$ and  $-\mu^2+\lambda^2+\eta^2=1$. For the item \emph{(2)}, ${({h^{S}_t})}_{v}(s)={({h^{S}_t})}_{v}'(s)=0$ if and only if
% $v=\mu n_{\gamma}(s)+\lambda n_{1}(s)+\eta n_2(s)$ with $-\mu^2+\lambda^2+\eta^2=1$  and $\langle t'(s),v\rangle= -\mu k_n+\lambda k_g=0$. This is equivalent to
% $$v=\frac{\cos \theta}{\sqrt{k_n^2(s)-k_g^2(s)}}\left(k_g(s)n_{\gamma}(s)+k_n(s)n_1(s)\right)+\sin \theta n_2(s).$$
%
%For the item \emph{(3)}, ${({h^{S}_t})}_{v}(s)={({h^{S}_t})}_{v}'(s)={({h^{S}_t})}_{v}''(s)=0$ if and only if
% $$v=\frac{\cos \theta}{\sqrt{k_n^2(s)-k_g^2(s)}}\left(k_g(s)n_{\gamma}(s)+k_n(s)n_1(s)\right)+\sin \theta n_2(s) \,\,\,\mbox{and}\,\,\, \langle t''(s),v\rangle=0.$$ As $t''(s)=(k_n^2(s)-k_g^2(s))t(s)+(k'_n(s)+k_g(s)\tau_1(s))n_{\gamma}(s)+(k_n(s)\tau_1(s)+k'_g(s))n_1(s)+(k_n(s)\tau_2(s)+k_g(s)\tau_g(s))n_2(s)$,
% we have that the above assertion is equivalent to  $$v=\frac{\cos \theta}{\sqrt{k_n^2(s)-k_g^2(s)}}\left(k_g(s)n_{\gamma}(s)+k_n(s)n_1(s)\right)+\sin \theta n_2(s)$$ and $\tan \theta=\dfrac{k_gk'_n+k_g^2\tau_1-k_n^2\tau_1-k_nk'_g}{(k_n\tau_2+k_g\tau_g)\sqrt{k_n^2-k_g^2}}(s)$.
%
%For realize the calculations of the items \emph{(4)} and \emph{ (5)} we use the Frenet-Serret type formulae of $\gamma$.  As the calculations are laborious and long we omit the details here.
%\end{proof}

Following  Proposition \ref{singhs}, we define the invariant \\$ \rho(s)=\Bigg((-k_gk''_n-k_gk_n\tau_2^2-2k_gk'_g\tau_1-k_g^2\tau'_1-k_g^2\tau_g\tau_2+2k_nk'_n\tau_1+k_n^2\tau'_1-k_n^2k_g\tau_2+k''_gk_n-k_gk_n\tau_g^2)
(k_n\tau_2+k_g\tau_g)+(k_gk'_n+k_g^2\tau_1-k_n^2\tau_1-k_nk'_g)(2k'_n\tau_2+k_n\tau_1\tau_g+k_n\tau'_2+2k'_g\tau_g+k_g\tau_1\tau_2+k_g\tau'_g)\Bigg)(s)$ \\of the curve $\gamma$.  We will study the geometric meaning of this invariant. Motivated by Proposition \ref{singhs}, we define the following surface and its singular locus.  Let $\gamma:I\rightarrow M$ be a  unit speed curve with $k_g(s)\neq0$, $k_n^2(s)> k_g^2(s)$ and $(k_n\tau_2+k_g\tau_g)(s)\neq0$, a  surface $DS_{\gamma}:I\times J\rightarrow S^3_1$ is defined by
$$DS_{\gamma}(s,\theta)=\frac{\cos \theta}{\sqrt{k_n^2(s)-k_g^2(s)}}\left(k_g(s)n_{\gamma}(s)+k_n(s)n_1(s)\right)+\sin \theta n_2(s),$$ where $J=[0,2\pi]$.
We call $DS_{\gamma}$ a \emph{ de Sitter surface} of $\gamma$. We now define $DC_{\gamma}=DS_{\gamma}(s,\theta(s))$, where $\tan\theta(s)=\dfrac{k_gk'_n+k_g^2\tau_1-k_n^2\tau_1-k_nk'_g}{\sqrt{k_n^2-k_g^2}(k_n\tau_2+k_g\tau_g)}(s)$.  We call  $DC_{\gamma}$ a \emph{de Sitter curve} of $\gamma$.  By Theorem \ref{teo:classificacao2} (1), this curve is the locus of the singular points of the de Sitter surface of $\gamma$

\begin{cor}\label{coro:discsupdesitter}
 The de Sitter surface of  $\gamma$ is the discriminant set $\mathcal{D}_{H_t^S}$  of the family of spacelike tangential height functions  $H^{S}_t$.
\end{cor}
\begin{proof}
The proof follows from the definition of the discriminant set given in the Section \ref{sec:pre} and by  Proposition  \ref{singhs} (2).
\end{proof}

\begin{prop}\label{desdobramentohtS}
Let $\gamma:I\rightarrow M$ be a  unit speed curve with $k_g(s)\neq0$ and $(k_n\tau_2+k_g\tau_g)(s)\neq0$.
\begin{itemize}
  \item [(a)] If $(h^{S}_t)_{v_0}$ has an $A_2$-singularity at $s_0$, then $H^{S}_t$ is a versal deformation of $(h^{S}_t)_{v_0}$.
  \item [(b)] If $(h^{S}_t)_{v_0}$ has an $A_3$-singularity  at $s_0$ and $\lambda_0(s_0)\neq0$ (which is a generic condition), then $H^{S}_t$ is a versal deformation of $(h^{S}_t)_{v_0}$.
\end{itemize}

 %In other words, $H^{S}_t$ is generically  a versal deformation of $(h^{S}_t)_{v_0}$.
\end{prop}
In the case of the de Sitter surface we have an analogous result to Proposition \ref{prop:defnova}, just considering the deformation $\widetilde{H}:I\times S^3_1\times \mathbb{R}\rightarrow \mathbb{R}$ by $\widetilde{H}(s,v,u)=H^{S}_{t}(s,v)+u(s-s_0)^2=\langle t(s),v\rangle+u(s-s_0)^2.$
\begin{prop}\label{prop:defnova2}
If $(h^{S}_t)_{v_0}$  has an $A_3$-singularity at $s_0$ and $\lambda_0(s_0)=0$, then $\widetilde{H}$ is a  versal deformation of $(h^{S}_t)_{v_0}$.
\end{prop}

The  Propositions \ref{desdobramentohtS} and \ref{prop:defnova2} give us the following result.

\begin{theo}\label{teo:classificacao2}
Let $\gamma:I\rightarrow M$ be a  unit speed curve  such that $k_g(s)\neq0$, $k_n^2(s)> k_g^2(s)$ and  $(k_n\tau_2+k_g\tau_g)(s)\neq0$, and $DS_{\gamma}$ the  de Sitter surface of $\gamma$. Then we have the following:
\begin{itemize}
\item[(1)]  $DS_{\gamma}$ is singular at $(s_0, \theta_0)$ if and only if $$\tan \theta_0=\dfrac{k_gk'_n+k_g^2\tau_1-k_n^2\tau_1-k_nk'_g}{\sqrt{k_n^2-k_g^2}(k_n\tau_2+k_g\tau_g)}(s_0).$$
    That is, the singular points of the de Sitter surface are given by  $DS_{\gamma}(s)=DS_{\gamma}(s,\theta(s))$, where $\tan \theta(s)$ satisfies the above equation.
\item[(2)]    The germ of $DS_{\gamma}$ at $(s_0, \theta_0)$ is locally diffeomorphic to the cuspidal edge  if
     $$\tan \theta_0=\dfrac{k_gk'_n+k_g^2\tau_1-k_n^2\tau_1-k_nk'_g}{(k_n\tau_2+k_g\tau_g)\sqrt{k_n^2-k_g^2}}(s_0)\,\,\,\hbox{and}\,\, \rho(s_0)\neq0.$$
\item[(3)]  The germ of $DS_{\gamma}$ at $(s_0, \theta_0)$ is  locally diffeomorphic to the swallowtail if
       $$\tan \theta_0=\dfrac{k_gk'_n+k_g^2\tau_1-k_n^2\tau_1-k_nk'_g}{(k_n\tau_2+k_g\tau_g)\sqrt{k_n^2-k_g^2}}(s_0),\,\,\,\lambda_0(s_0)\neq0,\,\, \rho(s_0)=0\,\, \hbox{and}\,\, \rho'(s_0)\neq0.$$

    \item[(4)] The germ of $DS_{\gamma}$ at $(s_0, \theta_0)$ is diffeomorphic to a cuspidal beaks if    $$
    \lambda_0(s_0)=0,\,\,\,\lambda_1(s_0)\neq0,\,\,\,\rho(s_0)=0\,\,\, \mbox{and}\,\,\,\rho'(s_0)\neq0.$$

\item[(5)] A cuspidal lips does not appear.
\end{itemize}
\end{theo}

In the next proposition we relate the curve $\gamma$ of the de Sitter surface with the invariant $\rho$ and a timelike slice surface. In this case, the singular locus of the de Sitter surface of $\gamma$ is a point.

\begin{prop}\label{pro:slice2}
Let $\gamma:I\rightarrow M$ be a  unit speed curve such that $k_g(s)\neq0$, $(k_n\tau_2+k_g\tau_g)(s)\neq0$ and $k_n^2(s)>k_g^2(s)$ for any $ s\in I$.  Let  $DS_{\gamma}(s,\theta(s))$ be the singular points of the de Sitter surface of $\gamma$. Then the following conditions are equivalent:
\begin{itemize}
\item[(1)] $DS_{\gamma}(s,\theta(s))$ is a constant spacelike vector;
\item[(2)] $\rho(s)\equiv 0$;
\item[(3)] there exist a spacelike vector $v$ and a real number $c$ such that  $Im(\gamma)\subset M\cap HP(v,c)$.
\end{itemize}
\end{prop}

In the previous  result the invariant $\rho\equiv0$ means that the curve $\gamma$ is a part of a timelike  slice surface. For the next results we assume that $\rho\not\equiv0$, that is $\gamma$ is not part of any timelike slice surface $M\cap HP(v,c)$.
%\begin{defn}\label{defcontact}
%Let $F:\mathbb{R}^{4}_1\rightarrow \mathbb{R}$ (respectively, $F\mid_{M}:M\rightarrow \mathbb{R}$) be a submersion and $\gamma:I\rightarrow M$ be a regular curve. We say that $\gamma$ and $F^{-1}(0)$ (respectively $F^{-1}(0)\cap M$) have contact of order $k$ at $s_0$ if the function  $f(s)=F\circ\gamma(s)$ satisfies $f(s_0)=f'(s_0)=\cdots=f^{(k)}(s_0)=0$ and $f^{(k+1)}(s_0)\neq0$, i.e., $f$ has $A_k$-type singularity at $s_0$.
%\end{defn}
%We define $C(2,3,4)=\{(t^2,t^3,t^4)\,\mid\,t\in\mathbb{R}\}$, which is called a $(2,3,4)$-\emph{cusp}. We have the following result.
\begin{prop}\label{prop:contato2}
Let $\gamma:I\rightarrow M$ be a  unit speed curve  such that $k_g(s)\neq0$, $(k_n\tau_2+k_g\tau_g)(s)\neq0$ and $k_n^2(s)>k_g^2(s)$ for any $s\in I$. Let $v_0=DS_{\gamma}(s_0,\theta_0)$ and $c=\langle \gamma(s_0),v_0\rangle$. Then we have the following:
\begin{itemize}
\item[(1)] $\gamma$ and the timelike slice surface $M\cap HP(v_0,c)$ have contact of at least order 3 at $s_0$ if and only if ${({h^{S}_t})}_{v_0}$ has $A_k$-singularity at $s_0$, $k\geq2$. Furthermore, if  $\gamma$ and the timelike slice surface $M\cap HP(v_0,c)$ have contact of  order exactly 3 at $s_0$, then the  de Sitter curve $DC_{\gamma}$ of $\gamma$ is, at $s_0$, locally diffeomorphic to a line at  $s_0$.
\item[(2)] $\gamma$ and the timelike slice surface $M\cap HP(v_0,c)$  have contact of order 4 at $s_0$ if and only if $ {({h^{S}_t})}_{v_0}$ has $A_{3}$-singularity at $s_0$. In this case, if $\lambda_0(s_0)\neq0$, then the  de Sitter curve $DC_{\gamma}$ of $\gamma$ is, at $s_0$, locally diffeomorphic to the $(2,3,4)$-\emph{cusp} $C(2,3,4)$.
\end{itemize}
\end{prop}

\section{Examples}\label{sec:exas3}
In this section, we consider two examples of curves on spacelike hypersurface  $M$ in $\mathbb{R}^4_1$. One of them is $M=\mathbb{R}^3$ another is  $H^3(-1)$.

\begin{ex}\rm
We consider $M=\mathbb{R}^3$, $\gamma:I\rightarrow \mathbb{R}^3$,  the Frenet  frame $\{t,n,b\}$  and the Frenet-Serret formulae as in Example \ref{ex:1}.
$$
\left\{
  \begin{aligned}
     e_{0}'(s) &=0, \\
     t'(s) & =k(s)\,n(s),\\
 n'(s) &= -k(s)\,t(s)+\tau(s)\,b(s),\\
  b'(s) &= -\tau(s)\,n(s).
  \end{aligned}
\right.$$

In this case, the de Sitter surface of  $\gamma$ in $S^3_1\subset \mathbb{R}^4_1$ can not be  defined.
\end{ex}

\begin{ex}\rm
We consider $M=H^3(-1)$, $\gamma:I\rightarrow H^3(-1)$ and the pseudo orthonormal frame $\{\gamma, t,n_1,n_2\}$ as in Example \ref{ex:2}.
$$
\left\{
  \begin{aligned}
     \gamma'(s) &=t(s), \\
     t'(s) & =\gamma(s)+k_h(s)\,n_1(s),\\
 n_1'(s) &= -k_h(s)\,t(s)+\tau_h(s)\,n_2(s),\\
  n_2'(s) &= -\tau_h(s)\,n_1(s).
  \end{aligned}
\right.$$
Therefore, for $k_h^2(s)< 1$ the de Sitter surface of $\gamma$ is given by
 $$DS_{\gamma}(s,\theta)=\dfrac{\cos\theta }{\sqrt{1-k_h^2(s)}}(k_h(s)\gamma(s)+n_1(s))+\sin \theta n_2(s).$$
 It follows that the de Sitter surface is precisely the de Sitter  focal surface of  $\gamma$ given in \cite{izumyiafs}.
\end{ex}

{\small
\par\noindent
Shyuichi Izumiya, Department of Mathematics, Hokkaido University, Sapporo 060-0810, Japan
\par\noindent
e-mail:{\tt izumiya@math.sci.hokudai.ac.jp}
\bigskip
\par\noindent
Ana Claudia Nabarro, Andrea de Jesus Sacramento,  Departamento de Matem\'{a}tica,
ICMC Universidade de S\~{a}o Paulo, Campus de S\~{a}o Carlos, Caixa Postal 668, CEP 13560-970, S\~{a}o Carlos-SP, Brazil
\par\noindent
e-mail:{\tt anaclana@icmc.usp.br }

\par\noindent e-mail:{\tt anddyunesp@yahoo.com.br}

}

%\bibliographystyle{acm}

%\bibliography{Andrea}

\end{document}